\documentclass[a4paper, 12pt]{article}
\usepackage{mathrsfs}
\usepackage{amsmath}

\usepackage{amsthm}
\usepackage{array}
\usepackage{cases}
\makeatletter
\def\th@plain{%
  \upshape 
}
\makeatother

\makeatletter
\renewenvironment{proof}[1][\proofname]{\par
  \pushQED{\qed}%
  \normalfont \topsep6\p@\@plus6\p@\relax
  \trivlist
  \item[\hskip\labelsep
        \bfseries
    #1\@addpunct{.}]\ignorespaces
}{%
  \popQED\endtrivlist\@endpefalse
}
\makeatother

\newtheorem{thm}{Theorem}[section]
\newtheorem{cor}[thm]{Corollary}

\newtheorem{lem}[thm]{Lemma}

\newtheorem{conj}[thm]{Conjecture}

\numberwithin{equation}{section}

\newcommand{\li}{\mathcal{L}}

\usepackage[varg]{txfonts}
\usepackage{graphicx, epsfig, subfigure}
\usepackage{enumerate} 
\usepackage[square, numbers, sort&compress]{natbib} 
\ifx\pdfoutput\undefined
 \usepackage[dvipdfm,%
  pdfstartview=FitH, 
  bookmarks=true,%
  bookmarksnumbered=true, 
  bookmarksopen=true, 
  plainpages=false,%
  pdfpagelabels,%
  colorlinks=true, 
  linkcolor=blue, 
  citecolor=blue,%
  urlcolor=black,
  pdfborder=001]{hyperref}
  \AtBeginDvi{}  
\else
 \usepackage[pdftex,%
  pdfstartview=FitH, 
  bookmarks=true,%
  bookmarksnumbered=true, 
  bookmarksopen=true, 
  plainpages=false,%
  pdfpagelabels,%
  colorlinks=true, 
  linkcolor=blue, 
  citecolor=blue,%
  urlcolor=black,
  pdfborder=001]{hyperref}
\fi

\numberwithin{equation}{section}

\begin{document}
\title{Group edge choosability of planar graphs without adjacent short cycles\footnotetext{Email addresses: sdu.zhang@yahoo.com.cn (X. Zhang), gzliu@sdu.edu.cn (G. Liu)}\thanks{This research is supported by NSFC (10971121, 61070230), RFDP(20100131120017) and GIIFSDU (yzc10040).}}
\author{Xin Zhang and Guizhen Liu\thanks{Corresponding author.}\\[.5em]
{\small School of Mathematics, Shandong University, Jinan 250100, P. R. China}\\
}
\date{}
\maketitle

\begin{abstract}
In this paper, we aim to introduce the group version of edge coloring and list edge coloring, and prove that all 2-degenerate graphs along with some planar graphs without adjacent short cycles is group $(\Delta(G)+1)$-edge-choosable
while some planar graphs with large girth and maximum degree is group $\Delta(G)$-edge-choosable.
\\[.5em]

\end{abstract}

\section{Introduction}

All graphs considered in this paper are finite, simple and undirected. We use $V(G)$, $E(G)$, $\delta(G)$ and $\Delta(G)$ to denote the vertex set, the edge set, the minimum degree and the maximum degree of a graph $G$. By $d_G(v)$, we denotes the degree of $v$ in $G$.
For a plane graph $G$, $F(G)$ denotes its face set and $d_G(f)$ denotes the degree of a face $f$ in $G$. The girth $g(G)$ of a graph $G$ is the length of its smallest cycle or $+\infty$ if $G$ is a forest. Throughout this paper, a $k$-, $k^+$- and $k^-$-vertex (resp. face) is a vertex (resp. face) of degree $k$, at least $k$ and at most $k$.
An $i$-alternating cycles in a graph $G$ is a cycle of even length in which alternate vertices have degree $i$. We say a graph $G$ is $k$-degenerate if $\delta(H)\leq k$ for every subgraph $H\subseteq G$.
Any undefined notation follows that of Bondy and Murty \cite{Bondy}.

In 1992, Jaeger \emph{et al.} \cite{Jaeger1992} introduced a concept of group connectivity as an generalization of nowhere zero flows and its dual concept group coloring. They proposed the definition of group colorability of graphs as the equivalence of group connectivity of $M$, where $M$ is a cographic matroid.
Let $G$ be a graph and $A$ be an Abelian group. Denote $F(G,A)$ to be the set of all functions $f: E(G)\mapsto A$ and $D$ to be an arbitrary orientation of $E(G)$. We say $G$ is $A$-colorable under the orientation $D$ if for any function $f\in F(G,A)$, $G$ has an $(A,f)$-coloring, namely, a vertex coloring $c: V(G)\mapsto A$ such that $c(u)-c(v)\neq f(uv)$ for every directed edge $uv$ from $u$ to $v$. In \cite{Lai}, Lai and Zhang presented that for any Abelian group $A$, a graph $G$ is $A$-colorable under the orientation $D$ if and only if $G$ is $A$-colorable under every orientation of $E(G)$. That is to say, the group colorability of a graph is independent of the orientation of $E(G)$. The group chromatic number of a graph $G$, denoted by $\chi_g(G)$, is defined to be the minimum $m$ for which $G$ is $A$-colorable for any Abelian group $A$ of order at least $m$. Clearly, $\chi(G)\leq \chi_g(G)$, where $\chi(G)$ is the chromatic number of $G$. Lai and Zhang \cite{Lai2002} proved that $\chi_g(G)\leq 5$ for every planar graph $G$ and Kr\'al' \emph{et al.} \cite{Kral2005} constructed a planar graph with the group chromatic number five. This implies the well-known Four-Colors Theorem for ordinary colorings can not be extended to group colorings. Nevertheless, some theorems for ordinary vertex colorings, such as Brooks' Theorem, still can be extended. The following theorem is due to Lai \emph{et al.} \cite{Lai2007}.

\begin{thm}\label{thm:brook}
For any connected simple graph $G$, $\chi_g(G)\leq \Delta(G)+1$, where equality holds if and only if $G$ is either a cycle or a complete graph.
\end{thm}

\noindent Here notice that for an even cycle $C_{2n}$, we have $\chi_g(C_{2n})=3$ by Theorem \ref{thm:brook} but $\chi(C_{2n})=2$.

In 2004, Kr\'al' and Nejedl\'y \cite{Kral2004} considered list group coloring as an extension of list coloring and group coloring. Let $G$ be a graph, $A$ be an Abelian group of order at least $k$ and $L: V(G)\mapsto 2^A$ be a $k$-uniform list assignment of $V(G)$. Denote $F(G,A)$ to be the set of all functions $f: E(G)\mapsto A$ and $D$ to be an arbitrary orientation of $E(G)$. We say $G$ is group $k$-choosable under the orientation $D$ if for any function $f\in F(G,A)$, $G$ has an $(A,L,f)$-coloring, that is an $(A,f)$-coloring $c$ such that $c(v)\in L(v)$ for every $v\in V(G)$. Note that the choice of an orientation of edges of $G$ is either not essential in this definition. The group choice number of a graph $G$, denoted by $\chi_{gl}(G)$, is defined to be the minimum $k$ for which $G$ is group $k$-choosable.
In \cite{Kral2004}, the authors showed that $\chi_{gl}(G)=2$ if and only if $G$ is a forest. Omidi \cite{Omidi2010} proved the group choice number of a graph without $K_5$-minor or $K_{3,3}$-minor and with girth at least 4 (resp. 6) is at most 4 (resp. 3). In \cite{Chuang}, Chuang \emph{et al.} also established the group choosability version of Brooks' Theorem, which extends Theorem \ref{thm:brook}.

\begin{thm}\label{thm:brook.list}
For any connected simple graph $G$, $\chi_{gl}(G)\leq \Delta(G)+1$, where equality holds if and only if $G$ is either a cycle or a complete graph.
\end{thm}

In this paper, we aim to introduce a group version of edge coloring and list edge coloring. Recall that the line graph of a graph $G$, denoted by $\mathcal{L}(G)$, is a graph such that each vertex of $\mathcal{L}(G)$ represents an edge of $G$ and two vertices of $\mathcal{L}(G)$ are adjacent if and only if their corresponding edges share a common endpoint in $G$. For an edge $uv\in E(G)$, we use $e_{uv}$ to denote the vertex in $\li(G)$ that represents $uv$ in $G$. Clearly, the edge chromatic number $\chi'(G)$ of a graph $G$ is equal to the vertex chromatic number $\chi(\mathcal{L}(G))$ of its line graph $\mathcal{L}(G)$.
In view of this, the group version of edge coloring and list edge coloring can be defined naturally. For an Abelian group $A$ of order at least $k$, we say $G$ is group $A$-edge-colorable if $\mathcal{L}(G)$ is group $A$-colorable and say $G$ is group $k$-edge-choosable if $\mathcal{L}(G)$ is group $k$-choosable.
By $\chi'_{g}(G)=\chi_g(\mathcal{L}(G))$ and $\chi'_{gl}(G)=\chi_{gl}(\mathcal{L}(G))$, we denotes the group edge chromatic number and the group edge choice number of a graph $G$. First of all, we have the following basic theorem.

\begin{thm}\label{thm:basic.edge}
For any connected simple graph $G$, $$\Delta(G)\leq \chi'_{g}(G)\left\{
                                                \begin{array}{ll}
                                                  =\chi'_{gl}(G)=2, & \hbox{if $G$ is a path;} \\
                                                  =\chi'_{gl}(G)=3, & \hbox{if $G$ is a cycle;} \\
                                                  \leq \chi'_{gl}(G)\leq 2\Delta(G)-2, & \hbox{if $\Delta(G)\geq 3$.}
                                                \end{array}
                                              \right.
$$
\end{thm}

\begin{proof}
Since $\chi'_{gl}(G)\geq \chi'_g(G)\geq \chi'(G)\geq \Delta(G)$, the left inequality in above theorem holds. If $G$ is a path (resp. cycle), then $\mathcal{L}(G)$ is also a path (resp.  cycle). So by Theorems \ref{thm:brook} and \ref{thm:brook.list}, we have $\chi'_g(G)=\chi'_{gl}(G)=2~({\rm resp.}~3)$. If $\Delta(G)\geq 3$, then $G$ is neither a cycle nor a star, which implies $\mathcal{L}(G)$ is neither a cycle nor a complete graph. So $\chi'_{g}(G)\leq \chi'_{gl}(G)=\chi_{gl}(\mathcal{L}(G))\leq \Delta(\mathcal{L}(G))\leq 2\Delta(G)-2$ by Theorems \ref{thm:brook} and \ref{thm:brook.list}.
\end{proof}

\noindent From Theorem \ref{thm:basic.edge}, we can find that $\chi'_{g}(G)\leq \chi'_{gl}(G)\leq \Delta(G)+1$ for every graph with maximum degree 3 and $\chi'_{g}(G)=\chi'_{gl}(G)$ for every graph with maximum degree 2. These evidences motivate us to conjecture the analogue of Vizing's Theorem on edge chromatic number and list edge coloring Conjecture on edge choice number.

\begin{conj}\label{conj:Vizing.group}
For any simple graph $G$, $\Delta(G)\leq \chi'_g(G)\leq \Delta(G)+1$.
\end{conj}

\begin{conj}\label{conj:list.group}
For any simple graph $G$, $\chi'_g(G)=\chi'_{gl}(G)$.
\end{conj}

In the next section, we will confirm Conjecture \ref{conj:Vizing.group} for all 2-degenerate graphs and some planar graphs without adjacent short cycles and confirm Conjecture \ref{conj:list.group} for some planar graphs with large girth and maximum degree.

For a nonnegative integer $i$, we call a graph $G$ is group $(\Delta(G)+i)$-edge-critical if $\chi'_{gl}(G)>\Delta(G)+i$ but $\chi'_{gl}(H)\leq \Delta(H)+i$ for every proper subgraph $H\subset G$. The $(\Delta(G)+i)$-edge-critical graph in terms of list edge coloring can be defined similarly. In most of the articles concerning list $(\Delta+1)$-edge coloring of planar graphs in the literature such as \cite{Cohen} and \cite{Hou2009}, it was proved and essential that a $3$-alternating cycle $C$ can not appear in a $(\Delta+1)$-edge-critical graph $G$ because if such a cycle $C$ do exist, then $G-E(C)$ is $(\Delta+1)$-edge choosable and every edge of $C$ has at least two available colors since it is incident with $\Delta(G)+1$ edges, of which $\Delta(G)-1$ are colored, which implies that one can extend the list $(\Delta+1)$-edge coloring of $G-E(C)$ to $G$ by the fact that even cycles are 2-edge-choosable. However, this technique is invalid for group edge choosability since any cycle is not group $2$-edge-choosable by Theorem \ref{thm:basic.edge}.

\section{Main results and their proofs}

We begin with this section by proving an useful Lemma, which will be frequently used in the next proofs and implies Conjecture \ref{conj:Vizing.group} holds for all 2-degenerate graphs.

\begin{lem}\label{lem:degree.sum}
Let $i$ be a nonnegative integer and $G$ be a group $(\Delta(G)+i)$-edge-critical graph. Then $G$ is connected and $d_G(u)+d_G(v)\geq \Delta(G)+i+2$ for any edge $uv\in E(G)$.
\end{lem}

\begin{proof}
The connectivity of $G$ directly follows from its definition. Suppose there is an edge $uv\in E(G)$ such that $d_G(u)+d_G(v)\leq \Delta(G)+i+1$. Then for an Abelian group $A$ of order at least $\Delta(G)+i$,  a $(\Delta(G)+i)$-uniform list assignment $L: V(\mathcal{L}(G))\mapsto 2^A$  and a function $f\in F(\mathcal{L}(G),A)$, $\mathcal{L}(G)$ is not $(A,L,f)$-colorable but $\mathcal{L}(G-uv)$ is. Let $c$ be an $(A,L,f)$-coloring of $\mathcal{L}(G-uv)$. Notice that now in $\mathcal{L}(G)$ the only uncolored vertex under $c$ is $e_{uv}$, which is adjacent to $m=d_G(u)+d_G(v)-2\leq \Delta(G)+i-1$ colored vertices, say $e_1,e_2,\cdots,e_m$. Without any loss of generality we assume $e_{uv}$ is the head of each edge $e_ie_{uv}$ in $\li(G)$ under a given orientation $D$ of $E(\li(G))$, where $1\leq i\leq m$. Now assign $e_{uv}$ a color in $S= L(e_{uv})-\bigcup_{i=1}^m\{c(e_i)-f(e_ie_{uv})\}$. Notice that $|S|\geq \Delta(G)+i-m\geq 1$. So we have extended $c$ to an $(A,L,f)$-coloring of $\li(G)$. This implies $G$ is group $(\Delta(G)+i)$-edge-choosable, a contradiction.
\end{proof}

\begin{cor}\label{cor:mini.degree}
Let $i$ be a nonnegative integer and $G$ be a group $(\Delta(G)+i)$-edge-critical graph. Then $\delta(G)\geq i+2$.
\end{cor}

\begin{cor}\label{cor:2-degenerate}
Every 2-degenerate graph is group $(\Delta(G)+1)$-edge-choosable.
\end{cor}

\begin{thm} \label{thm:planar}
Let $G$ be a planar graph such that $G$ does not contain an $i$-cycle adjacent to a $j$-cycle where $3\leq i\leq s$ and $3\leq j\leq t$. If\\
\indent(1) $s=3$, $t=3$ and $\Delta(G)\geq 8$, or\\
\indent(2) $s=3$, $t=4$ and $\Delta(G)\geq 6$, or\\
\indent(3) $s=4$, $t=5$ and $\Delta(G)\geq 5$, or \\
\indent(4) $s=4$, $t=7$,\\
then $G$ is group $(\Delta(G)+1)$-edge-choosable.
\end{thm}

\begin{proof}
The proof is carried out by contradiction and discharging. Suppose $G$ is a minimum counterexample to the theorem. Then by Lemma \ref{lem:degree.sum}, one can easily find that $G$ is a connected and group $(\Delta(G)+1)$-edge-critical planar graph with $\delta(G)\geq 3$.

By Euler's Formula, for any $n>2m>0$, we have
\begin{equation}\label{eq:euler}
\sum_{v\in V(G)}[(\frac{n}{2}-m)d_G(v)-n]+\sum_{f\in F(G)}(md_G(f)-n)=-2n<0.
\end{equation}
Assign each vertex $v\in V(G)$ an initial charge $c(v)=(\frac{n}{2}-m)d_G(v)-n$ and each face $f\in F(G)$ an initial charge $c(f)=md_G(f)-n$. Then by (\ref{eq:euler}), we have $\sum_{x\in V(G)\cup F(G)}c(x)<0$. To prove the theorem, we are ready to construct a new charge function $c'$ on $V(G)\cup F(G)$ according some defined discharging rules, which only move charge around but do not affect the total charges, so that after discharging the final charge $c'(x)$ of each element $x\in V(G)\cup F(G)$ is nonnegative. This contradiction completes the proof of the theorem in final. In the following, we call a face $f\in F(G)$ is simple if the boundary of $f$ is a cycle and denote $m_v(f)$ to be the number of times through $v$ by a face $f$ in clockwise order. Obviously, if $v$ is a non-cut vertex or $f$ is a simple face, then $m_v(f)=1$.

(1) Let $S$ be the set of 3-vertices, 4-vertices and 5-vertices in $G$. By Lemma \ref{lem:degree.sum}, we can claim that $S$ forms an independent set in $G$ since $\Delta(G)\geq 8$. Now we choose $m=2$ and $n=6$ in (\ref{eq:euler}) and define the discharging rules as follows:\\
\indent \textbf{R1.1}. From each $4^+$-face $f$ to its incident vertex $v\in S$, transfer $m_v(f)$.\\
\indent \textbf{R1.2}. From each $8^+$-vertex $u$ to its adjacent 3-vertex $v$, transfer $\frac{1}{2}$ if $uv$ is incident with a 3-cycle.\\
Without any loss of generality, we always assume $v$ is a non-cut vertex and $f$ is simple in the following arguments (because during the calculational part of discharging, the case when $v$ is a cut vertex that is incident with a non-simple face $f$ is equivalent to the case when $v$ is incident with $m_v(f)$ simple faces with the same degree of $f$, and the  case when $f$ is a non-simple face that is incident with a cut vertex $v$ is equivalent to the case when $f$ is incident with $m_v(f)$ non-cut vertices with the same degree of $v$). Suppose $d_G(v)=3$. Then by Lemma \ref{lem:degree.sum}, $v$ is adjacent to three $8^+$-vertices. If $v$ is incident with a 3-face, then $v$ is also incident with two $4^+$-faces since there are no two adjacent $3$-cycles in $G$. This implies $c'(v)\geq c(v)+2\times\frac{1}{2}+2\times 1=0$ by R1.1 and R1.2. If $v$ is not incident with any 3-faces, then $c'(v)\geq c(v)+3\times 1=0$ by R1.1. Suppose $4\leq d_G(v)\leq 5$. One can easy show that $v$ is incident with at least two $4^+$-faces, which implies $c'(v)\geq c(v)+2\times 1=0$. Suppose $6\leq d_G(v)\leq 7$. Then it is easy to see $w'(v)=w(v)\geq 0$. Suppose $d_G(v)\geq 8$. Notice that any two 3-cycles are not adjacent in $G$, so $G$ is incident with at most $\lfloor\frac{d_G(v)}{2}\rfloor$ 3-faces, which implies $v$ may transfer charges to at most $\lfloor\frac{d_G(v)}{2}\rfloor$ 3-vertices by R1.2 since any 3-vertices are not adjacent in $G$ either. So we have $c'(v)\geq d_G(v)-6-\frac{1}{2}\lfloor\frac{d_G(v)}{2}\rfloor\geq 0$ for $d_G(v)\geq 8$. Suppose $d_G(f)=3$. Then it is trivial that $c'(f)=c(f)=0$. Suppose $d_G(f)\geq 4$. Then $f$ may transfer charges to at most $\lfloor\frac{d_G(f)}{2}\rfloor$ vertices by R1.1 since $S$ is an independent set in $G$. This implies $c'(f)\geq 2d_G(f)-6-\lfloor\frac{d_G(f)}{2}\rfloor\geq 0$ for $d_G(f)\geq 4$ in final.

(2) We choose $m=3$ and $n=10$ in (\ref{eq:euler}) and define the discharging rules as follows:\\
\indent \textbf{R2.1}. From each $6^+$-vertex to its adjacent 3-vertex, transfer $\frac{1}{3}$.\\
\indent \textbf{R2.2}. From each $4$-face $f$ to its incident vertex $v$, transfer $m_v(f)$ if $d_G(v)=3$, $\frac{1}{2}m_v(f)$ if $d_G(v)=4$.\\
\indent \textbf{R2.3}. From each $5^+$-face $f$ to its incident vertex $v$, transfer $\frac{3}{2}m_v(f)$ if $d_G(v)=3$, $m_v(f)$ if $d_G(v)=4$.\\
\indent \textbf{R2.4}. From each $5^+$-face to its adjacent $3$-face, transfer $\frac{1}{3}$.\\
Suppose $d_G(v)=3$. Then by Lemma \ref{lem:degree.sum}, $v$ is adjacent to three $6^+$-vertices since $\Delta(G)\geq 6$. If $v$ is incident with a 3-face, then $v$ is also incident with two $5^+$-face by the condition in the theorem. This implies $c'(v)\geq c(v)+3\times\frac{1}{3}+2\times\frac{3}{2}=0$ by R2.1 and R2.3. If $v$ is incident with no $3$-faces, then $v$ is incident with three $4^+$-face, which implies $c'(v)\geq c(v)+3\times\frac{1}{3}+3\times 1=0$ by R2.1, R2.2 and R2.3. Suppose $d_G(v)=4$. If $v$ is incident with a 3-face, then $v$ is incident with at least two $5^+$-faces, which implies $c'(v)\geq c(v)+2\times 1=0$ by R2.3. If $v$ is incident with no $3$-faces, then $v$ is incident with four $4^+$-faces, which implies $c'(v)\geq c(v)+4\times \frac{1}{2}=0$ by R2.2 and R2.3. Suppose $d_G(v)=5$. Then it is easy to see $c'(v)=c(v)=0$. Suppose $d_G(v)\geq 6$. Then by R2.1, we have $c'(v)\geq 2d_G(v)-10-\frac{1}{3}d_G(v)\geq 0$. Suppose $d_G(f)=3$. Then by the condition of the theorem $f$ is adjacent to three $5^+$-faces, implying $c'(f)\geq c(f)+3\times \frac{1}{3}=0$ by R2.4. Suppose $d_G(f)\geq 4$. Then $f$ is incident with at most $\lfloor\frac{d_G(f)}{2}\rfloor$ $4^-$-vertices since there is no adjacent $4^-$-vertices in $G$ by Lemma \ref{lem:degree.sum}. This implies $c'(f)\geq c(f)-2\times 1=0$ for $d_G(f)=4$ by R2.2, and $c'(v)\geq 3d_G(f)-10-\frac{1}{3}d_G(f)-\frac{3}{2}\lfloor\frac{d_G(f)}{2}\rfloor>0$ for $d_G(f)\geq 5$ by R2.3 and R2.4.

(3) We choose $m=2$ and $n=6$ in (\ref{eq:euler}) and define the discharging rules as follows:\\
\indent \textbf{R3.1}. From each $5$-face $f$ to its incident vertex $v$, transfer $m_v(f)$ if $d_G(v)=3$, $\frac{1}{2}m_v(f)$ if $d_G(v)=4$, $\frac{1}{5}m_v(f)$ if $d_G(v)=5$.\\
\indent \textbf{R3.2}. From each $6^+$-face $f$ to its incident vertex $v$, transfer $\frac{3}{2}m_v(f)$ if $d_G(v)=3$, $m_v(f)$ if $d_G(v)=4$, $\frac{1}{3}m_v(f)$ if $d_G(v)=5$.\\
Suppose $d_G(v)=3$. If $v$ is incident with a $4^-$-face, then $v$ is also incident with two $6^+$-faces by the condition of the theorem, which implies by R3.2 that $c'(v)\geq c(v)+2\times \frac{3}{2}=0$. If $v$ is incident with no $4^-$-faces, then by R3.1 and R3.2 we have $c'(v)\geq c(v)+3\times 1=0$. Suppose $d_G(v)=4$. If $v$ is incident with a $4^-$-face, then $v$ is incident with at least two $6^+$-faces, which implies $c'(v)\geq c(v)+2\times 1=0$ by R3.2. If $v$ is incident with no $4^-$-faces, then by R3.1 and R3.2 we also have $c'(v)\geq c(v)+4\times \frac{1}{2}=0$. Suppose $d_G(v)=5$. If $v$ is incident with at least one $4^-$-face, then $v$ is incident with either three $6^+$-faces implying $c'(v)\geq c(v)+3\times\frac{1}{3}=0$ by R3.2, or two $5^+$-faces and two $6^+$-faces implying $c'(v)\geq c(v)+2\times\frac{1}{5}+2\times\frac{1}{3}>0$ by R3.1 and R3.2. If $v$ is incident with no $4^-$-faces, then by R3.1 and R3.2 we still have $c'(v)\geq c(v)+5\times\frac{1}{5}=0$. Suppose $d_G(v)\geq 6$ or $3\leq d_G(f)\leq 4$. Then it is clear that $c'(v)=c(v)\geq 0$ and $c'(f)=c(f)\geq 0$. Suppose $d_G(f)=5$. If $f$ is incident with no 3-vertices, then by R3.1 we have $c'(f)\geq c(f)-5\times\frac{1}{2}>0$. If $f$ is incident with at leat one 3-vertex, note that any $3$-vertex can not be adjacent to a $4^-$-vertex in $G$ by Lemma \ref{lem:degree.sum}, so $f$ is also incident with at least two $5^+$-vertices. This implies $c'(f)\geq c(f)-2\times\frac{1}{5}-3\times 1>0$ by R3.1. Suppose $d_G(f)\geq 6$. Then we shall have $d_G(f)-n_3-n_4\geq n_3$ by Lemma \ref{lem:degree.sum} since $\Delta(G)\geq 5$, where $n_i$ denotes the number of $i$-vertices that are incident with $f$ in $G$. This implies by R3.2 that $c'(f)\geq 2d_G(f)-6-\frac{3}{2}n_3-n_4-\frac{1}{3}(d_G(f)-n_3-n_4)=d_G(f)-6-\frac{2}{3}(2n_3+n_4-d_G(f))+\frac{1}{6}n_3\geq 0$ in final.

(4) We shall assume $\Delta(G)\geq 4$ in this part because the cases when $\Delta(G)\leq 3$ have been proved in Theorem \ref{thm:basic.edge}. Now we also choose $m=2$ and $n=6$ in (\ref{eq:euler}) and define the discharging rules as follows:\\
\indent \textbf{R4.1}. From each face $f$ of degree between 5 and 7 to its incident vertex $v$, transfer $m_v(f)$ if $d_G(v)=3$, $\frac{1}{2}m_v(f)$ if $d_G(v)\geq 4$.\\
\indent \textbf{R4.2}. From each $8^+$-face $f$ to its incident vertex $v$, transfer $\frac{3}{2}m_v(f)$ if $d_G(v)=3$, $m_v(f)$ if $d_G(v)\geq 4$.\\
Note that the above discharging rules are highly similar to the ones in part (3). So by a same analysis as in the previous part, one can also check that $c'(v)\geq 0$ for all $v\in V(G)$ and $c'(f)\geq 0$ for $3\leq d_G(f)\leq 4$. Now we shall only consider $5^+$-faces. Note that any $3$-vertices can not be adjacent in $G$ by Lemma \ref{lem:degree.sum} because we have already assumes $\Delta(G)\geq 4$. Thus $n_3\leq \lfloor\frac{d_G(f)}{2}\rfloor$ for any $f\in F(G)$, where $n_3$ is defined similarly as in part (3). Suppose $5\leq d_G(f)\leq 7$. Then by R4.1, we can deduce that $c'(f)\geq 2d_G(f)-6-n_3-\frac{1}{2}(d_G(f)-n_3)\geq \frac{3}{2}d_G(f)-6-\frac{1}{2}\lfloor\frac{d_G(f)}{2}\rfloor\geq 0$. Suppose $d_G(f)\geq 8$. We still have $c'(f)\geq 2d_G(f)-6-\frac{3}{2}n_3-1\times (d_G(f)-n_3)\geq d_G(f)-6-\frac{1}{2}\lfloor\frac{d_G(f)}{2}\rfloor\geq 0$ by R4.2 in final. This completes the proof of the theorem.
\end{proof}

As an immediately corollary of Theorem \ref{thm:planar}, we have the following two results.

\begin{cor}
Every planar graph with girth $g(G)\geq 5$ is group $(\Delta(G)+1)$-edge-choosable.
\end{cor}

\begin{cor}
Every planar graph with girth $g(G)\geq 4$ and maximum degree $\Delta(G)\geq 6$ is group $(\Delta(G)+1)$-edge-choosable.
\end{cor}

Another interesting topic concerting group edge colorings and list group edge colorings is to determine which class of graphs satisfies $\chi'_g(G)=\chi'_{gl}(G)$. In view of this, we end this paper by proving the following theorem, which confirms Conjecture \ref{conj:list.group} for some planar graphs with large girth and maximum degree.

\begin{thm}
Let $G$ be a planar graph with maximum degree $\Delta(G)\geq \Delta\geq 3$. If $g(G)\geq 4+\lceil\frac{8}{\Delta-2}\rceil$, then $\chi'_g(G)=\chi'_{gl}(G)=\Delta(G)$.
\end{thm}

\begin{proof}
We just need to prove $\chi'_{gl}(G)=\Delta(G)$ here. Suppose, to the contrary, that $G$ is a group $\Delta(G)$-edge-critical graph. Let $c(v)=2d_G(v)-6$ if $v\in V(G)$ and $c(f)=d_G(f)-6$ if $f\in V(G)$. Then by (\ref{eq:euler}), we have $\sum_{x\in V(G)\cup F(G)}c(x)<0$. Now we redistribute the charge of the vertices and faces of $G$ according the following discharging rules:\\
\indent \textbf{R1}. From each vertex of maximum degree to its adjacent 2-vertex, transfer $2-\frac{6}{\Delta}$.\\
\indent \textbf{R2}. From each face $f$ to its incident 2-vertex $v$, transfer $(\frac{6}{\Delta}-1)m_v(f)$.\\
We shall get a contradiction by proving $c'(x)\geq 0$ for every $x\in V(G)\cup F(G)$, where $c'(x)$ is the final charge of the element $x$ after discharging. Suppose $d_G(v)=2$. Then by Lemma \ref{lem:degree.sum}, the two neighbors of $v$ shall be both $\Delta(G)$-vertices, which implies $c'(v)\geq c(v)+2\times (2-\frac{6}{\Delta})+2\times (\frac{6}{\Delta}-1)=0$ by R1 and R2. Suppose $3\leq d_G(v)\leq \Delta(G)-1$ (if exists). Then it is clear that $c'(v)=c(v)\geq 0$. Suppose $d_G(v)=\Delta(G)$. Then by R1, one can easily deduce that $c'(v)\geq 2\Delta(G)-6-\Delta(G)(2-\frac{6}{\Delta})\geq 0$ since $\Delta(G)\geq \Delta$. Suppose $f$ is a face in $G$. Similarly as in the proof of Theorem \ref{thm:planar}, without loss of generality, we can assume $f$ is simple. Then by Lemma \ref{lem:degree.sum}, $f$ is incident with at most $\lfloor\frac{d_G(f)}{2}\rfloor$ 2-vertices. This implies by R2 that $c'(f)\geq d_G(f)-6-(\frac{6}{\Delta}-1)\lfloor\frac{d_G(f)}{2}\rfloor\geq \frac{3\Delta-6}{2\Delta}g(G)-6\geq \frac{3\Delta-6}{2\Delta}\cdot\frac{4\Delta}{\Delta-2}-6=0$ in final.
\end{proof}


\end{document}